\documentclass[preprint,10pt]{elsarticle}

\usepackage{amssymb}
\usepackage{amsmath}
\usepackage[all]{xy}
\usepackage{latexsym}
\usepackage{mathtools}
\usepackage{amsthm}
\usepackage{a4wide}
\usepackage{color}
\usepackage{hyperref}
\usepackage[hyphenbreaks]{breakurl}
\input diagxy

\journal{}

\theoremstyle{plain}
  \newtheorem{thm}{Theorem}[section]
  \newtheorem{lem}[thm]{Lemma}
  \newtheorem{prop}[thm]{Proposition}
  \newtheorem{cor}[thm]{Corollary}
\theoremstyle{definition}

  \newtheorem{exmp}[thm]{Example}
  \newtheorem{rem}[thm]{Remark}

\DeclareMathOperator{\dom}{dom}
\DeclareMathOperator{\cod}{cod}

\DeclareMathOperator{\ob}{ob}

\makeatletter
\def\ps@pprintTitle{%
 \let\@oddhead\@empty
  \let\@evenhead\@empty
  \def\@oddfoot{\vbox{\hsize=\textwidth\scriptsize
  \copyright 2015. This manuscript version is made available under the CC-BY-NC-ND 4.0 license \url{http://creativecommons.org/licenses/by-nc-nd/4.0/}. The published version is available at \url{http://dx.doi.org/10.1016/j.topol.2015.12.020}.\\
  }}%
  \let\@evenfoot\@oddfoot}
\makeatother

\newcommand{\da}{\downarrow}
\newcommand{\Da}{\Downarrow}
\newcommand{\ua}{\uparrow}
\newcommand{\ra}{\rightarrow}

\newcommand{\lda}{\swarrow}
\newcommand{\rda}{\searrow}

\newcommand{\Lra}{\Longrightarrow}

\newcommand{\rat}{\!\rightarrowtail\!}
\newcommand{\oto}{{\to\hspace*{-3.1ex}{\circ}\hspace*{1.9ex}}}

\newcommand{\bv}{\bigvee}
\newcommand{\bw}{\bigwedge}

\newcommand{\dv}{\dashv}

\newcommand{\nat}{\natural}

\newcommand{\si}{\sigma}

\newcommand{\CA}{\mathcal{A}}
\newcommand{\CB}{\mathcal{B}}
\newcommand{\CC}{\mathcal{C}}
\newcommand{\CD}{\mathcal{D}}
\newcommand{\CE}{\mathcal{E}}

\newcommand{\CJ}{\mathcal{J}}

\newcommand{\CQ}{\mathcal{Q}}

\newcommand{\Bf}{{\bf f}}
\newcommand{\Bg}{{\bf g}}
\newcommand{\Bh}{{\bf h}}

\newcommand{\sI}{{\sf I}}

\newcommand{\sP}{{\sf P}}
\newcommand{\sPd}{{\sf P}^{\dag}}

\newcommand{\sS}{{\sf S}}

\newcommand{\sY}{{\sf Y}}

\newcommand{\Cat}{{\bf Cat}}
\newcommand{\CAT}{{\bf CAT}}
\newcommand{\CATB}{{\CAT\Da_c\CB}}
\newcommand{\CatB}{{\Cat\Da_c\CB}}

\newcommand{\Chu}{{\bf Chu}}
\newcommand{\CHU}{{\bf CHU}}

\newcommand{\DIS}{{\bf DIS}}

\newcommand{\Met}{{\bf Met}}
\newcommand{\Ord}{{\bf Ord}}

\newcommand{\QUANT}{{\bf QUANT}}
\newcommand{\Rel}{{\bf Rel}}

\newcommand{\Set}{{\bf Set}}

\newcommand{\SUP}{{\bf SUP}}

\newcommand{\QCAT}{\CQ\text{-}\CAT}
\newcommand{\QCHU}{\CQ\text{-}\CHU}
\newcommand{\QCat}{\CQ\text{-}\Cat}
\newcommand{\QChu}{\CQ\text{-}\Chu}

\newcommand{\co}{{\rm co}}

\newcommand{\op}{{\rm op}}

\newcommand{\dPhi}{\Phi^{\da}}
\newcommand{\uPhi}{\Phi_{\ua}}

\newcommand{\dPsi}{\Psi^{\da}}
\newcommand{\uPsi}{\Psi_{\ua}}
\newcommand{\dXi}{\Xi^{\da}}
\newcommand{\uXi}{\Xi_{\ua}}

\newcommand{\PB}{\sP\CB}
\newcommand{\PC}{\sP\CC}
\newcommand{\PD}{\sP\CD}
\newcommand{\PE}{\sP\CE}
\newcommand{\PJ}{\sP\CJ}

\newcommand{\PdB}{\sP^{\dag}\CB}

\newcommand{\PdD}{\sP^{\dag}\CD}
\newcommand{\PdE}{\sP^{\dag}\CE}
\newcommand{\PdJ}{\sP^{\dag}\CJ}
\newcommand{\sYd}{\sY^{\dag}}

\newcommand{\QB}{\CQ_{\CB}}

\begin{document}

\begin{frontmatter}



\title{Topological categories, quantaloids and Isbell adjunctions}


\author{Lili Shen\fnref{A}}
\ead{shenlili@yorku.ca}

\author{Walter Tholen\corref{cor}\fnref{A}}
\ead{tholen@mathstat.yorku.ca}

\address{Department of Mathematics and Statistics, York University, Toronto, Ontario, Canada, M3J 1P3}
\address{\rm Dedicated to Eva Colebunders on the occasion of her 65th birthday}

\cortext[cor]{Corresponding author.}
\fntext[A]{Partial financial assistance by the Natural Sciences and Engineering Research Council of Canada is gratefully acknowledged.}

\begin{abstract}
In fairly elementary terms this paper presents, and expands upon, a recent result by Garner by which the notion of topologicity of a concrete functor is subsumed under the concept of total cocompleteness of enriched category theory. Motivated by some key results of the 1970s, the paper develops all needed ingredients from the theory of quantaloids in order to place essential results of categorical topology into the context of quantaloid-enriched category theory, a field that previously drew its motivation and applications from other domains, such as quantum logic and sheaf theory.
\end{abstract}

\begin{keyword}
Concrete category \sep topological category \sep enriched category \sep quantaloid \sep total category \sep Isbell adjunction


\MSC[2010] 	 18D20 \sep 54A99 \sep 06F99

\end{keyword}

\end{frontmatter}


\section{Introduction}

Garner's \cite{Garner2014} recent discovery that the fundamental notion of \emph{topologicity} of a concrete functor may be interpreted as precisely total (co)completeness for categories enriched in a free quantaloid reconciles two lines of research that for decades had been perceived by researchers in the respectively fields as almost intrinsically separate, occasionally with divisive arguments. While the latter notion is rooted in Eilenberg's and Kelly's enriched category theory (see \cite{Eilenberg1966,Kelly1982}) and the seminal paper by Street and Walters \cite{Street1978} on totality, the former notion goes back to Br{\"u}mmer's thesis \cite{Brummer1971} and the pivotal papers by Wyler \cite{Wyler1971,Wyler1971a} and Herrlich \cite{Herrlich1974} that led to the development of \emph{categorical topology} and the categorical exploration of a multitude of new structures; see, for example, the survey by Colebunders and Lowen \cite{Lowen-Colebunders2001}. The purpose of this paper is to present, and expand upon, Garner's result in the most accessible terms for readers who may not necessarily be familiar with the extensive apparatus of enriched category theory and, in particular, the theory of quantaloid-enriched categories, as developed mostly in Stubbe's papers \cite{Stubbe2005,Stubbe2006}.

Given a (potentially large) family of objects $X_i\ (i\in I)$ in a concrete category $\CE$ over a fixed category $\CB$ (which most often is simply the category of sets) and of maps $|X_i|\to Y$ in $\CB$ (with $|X_i|$ denoting the underlying $\CB$-object of $X_i$), topologicity of the functor $|\text{-}|:\CE\to\CB$ precisely asks for the existence of a ``best'' $\CE$-structure on $Y$, called a final lifting of the given structured sink of maps. Such liftings are needed not just for the formation of, say, topological sums and quotients (or in the dual situation, of products and subspaces), but in fact belong to the topologist's standard arsenal when defining new spaces from old in many situations, under varying terminology, such as that of a``weak topology''.  By contrast, at first sight, total cocompleteness appears to be a much more esoteric notion, as it entails a very strong existence requirement for colimits of possibly large diagrams, in both ordinary and enriched category theory. The surprising interpretation of final liftings as (so-called weighted) colimits lies at the heart of Garner's discovery. It was made possible by his quite simple observation that concrete categories over $\CB$ may be seen as categories enriched over the free quantaloid generated by $\CB$ (see Rosenthal \cite{Rosenthal1991}), with the categories enriched in such bicategories being first defined by Walters \cite{Walters1981}.

Without assuming any a-priori background by the reader on quantaloids, we show in Section \ref{concrete} how concrete categories over $\CB$ naturally lead to the formation of the free quantaloid over $\CB$ and their interpretation as categories enriched in that quantaloid. Section \ref{sink_topological} shows how a given structured sink may, without loss of generality, always be assumed to be a presheaf, which then produces Garner's result immediately. In Section \ref{cofibred} we discuss quantaloidal generalizations of Wyler's \cite{Wyler1971a} approach to topological functors, presented as simultaneous fibrations and cofibrations with large-complete fibres, and we carefully compare these notions with their enriched counterparts, namely that of being tensored, cotensored and conically complete, adding new facts and counter-examples to the known theory. Sections \ref{weighted_colimit} to \ref{universality} give a quick tour of the basic elements of quantaloid-enriched category theory as needed for the presentation of the self-dual concept of totality which, when applied in the concrete context, reproduces Hoffmann's \cite{Hoffmann1972} self-duality result for topological functors. A key tool here is provided by B{\' e}nabou's distributors \cite{Benabou1973} which, roughly speaking, generalize functors in the same way as relations generalize maps. Based on the paper by Shen and Zhang \cite{Shen2013a}, in Sections \ref{total_Chu} and \ref{IPhi} we show that a category is total precisely when it appears as the category of the fixed objects under the so-called Isbell adjunction induced by a distributor~--- which, among other things, reproduces the MacNeille completion of an ordered set. We also extend the known \cite{Stubbe2013a} characterization of totality as injectivity w.r.t. fully faithful functors from small quantaloid-enriched categories to large ones (Section \ref{total_injective}). When applied to concrete functors, it reproduces the characterization of topologicity through diagonal conditions, as first considered in Hu{\v s}ek \cite{Husek1967}, completed in Br{\" u}mmer-Hoffmann \cite{Brummer1976}, and generalized in Tholen-Wischnewsky \cite{Tholen1979a}.

In this paper we employ no specific strict set-theoretical regime, distinguishing only between sets (``small'') and classes (``(potentially) large'') and adding the prefix ``meta'' to categories whose objects may be large, thus trusting that our setting may be accommodated within the reader's favourite foundational framework. In fact, often the formation of such metacategories may be avoided as it is undertaken only for notational convenience.

We thank the anonymous referee for several helpful remarks.

\section{Concrete categories as free-quantaloid-enriched categories} \label{concrete}

For a (``base'') category $\CB$ with small hom-sets, a category $\CE$ that comes equipped with a faithful functor
$$|\text{-}|:\CE\to\CB$$
is usually called \emph{concrete} (over $\CB$) \cite{Adamek1990}. Referring to arrows in $\CB$ as \emph{maps}, we may then call a map $f:|X|\to|Y|$ with $X,Y\in\ob\CE$ a \emph{morphism} (or more precisely, an \emph{$\CE$-morphism}) if there is $f':X\to Y$ in $\CE$ with $|f'|=f$. In other words, since
$$\CE(X,Y)\cong|\CE(X,Y)|\subseteq\CB(|X|,|Y|),$$
being a morphism is a \emph{property} of maps between $\CE$-objects. It is therefore natural to associate with $\CB$ a new category $\QB$, the objects of which are those of $\CB$, but the arrows in $\QB$ are sets of maps:
$$\QB(S,T)=\{\Bf\mid\Bf\subseteq\CB(S,T)\}$$
for $S,T\in\ob\CB$. With
$$\Bg\circ\Bf=\{g\circ f\mid f\in\Bf,g\in\Bg\},$$
$${\bf 1}_S=\{1_S\}$$
(for $\Bf:S\to T$, $\Bg:T\to U$), $\QB$ becomes a category and, in fact, a \emph{quantaloid}, i.e., a category whose hom-sets are complete lattices such that the composition preserves suprema in each variable. Indeed, as is well known, $\QB$ is the \emph{free quantaloid} generated by $\CB$:

\begin{prop} {\rm\cite{Rosenthal1991}}
The assignment $\CB\mapsto\QB$ defines a left adjoint to the forgetful 2-functor $\QUANT\to\CAT$, which forgets the ordered structure of a quantaloid.
\end{prop}

A category $\CE$ concrete over $\CB$ may now be completely described by
\begin{itemize}
\item a class $\ob\CE$ of objects;
\item a function $|\text{-}|:\ob\CE\to\ob\QB(=\ob\CB)$ sending each object $X$ in $\CE$ to its \emph{extent} $|X|$ in $\CB$;
\item a family $\CE(X,Y)\in\QB(|X|,|Y|)$ $(X,Y\in\ob\CE)$, subject to
      \begin{itemize}
      \item ${\bf 1}_{|X|}\subseteq\CE(X,X)$,
      \item $\CE(Y,Z)\circ\CE(X,Y)\subseteq\CE(X,Z)$  $(X,Y,Z\in\ob\CE)$.
      \end{itemize}
\end{itemize}

A (concrete) functor $F:\CE\to\CD$ between concrete categories over $\CB$ (that must commute with the respective faithful functors to $\CB$) may then be described by a function $F:\ob\CE\to\ob\CD$ with
\begin{itemize}
\item $|FX|=|X|$,
\item $\CE(X,Y)\subseteq\CD(FX,FY)$ $(X,Y\in\ob\CE)$.
\end{itemize}

An order\footnote{In this paper, ``order'' refers to a reflexive and transitive relation (usually called preorder), with no requirement for antisymmetry, unless explicitly stated.} between concrete functors $F,G:\CE\to\CD$ is given by
\begin{align*}
F\leq G&\iff\forall X\in\ob\CE:\ 1_{|X|}:|FX|\to|GX|\ \text{is a}\ \CD\text{-morphism}\\
&\iff\forall X\in\ob\CE:\ {\bf 1}_{|X|}\subseteq\CD(FX,GX),
\end{align*}
rendering the (meta)category
$$\CATB$$
of concrete categories over $\CB$ as a 2-category, with 2-cells given by order.

Above we have described categories and functors concrete over $\CB$ as categories and functors \emph{enriched in} $\QB$. In fact, for any quantaloid $\CQ$, a \emph{$\CQ$-category} $\CE$ may be defined precisely as above, by just trading $\QB$ for $\CQ$ and ``$\subseteq$'' for the order ``$\leq$'' of the hom-sets of $\CQ$, and likewise for a \emph{$\CQ$-functor} $F:\CE\to\CD$ and the order of $\CQ$-functors. With $\QCAT$ denoting the resulting 2-(meta)category of $\CQ$-categories and $\CQ$-functors one may now formally confirm the equivalence of the descriptions given above, as follows:

\begin{prop} {\rm\cite{Garner2014}}
$\CATB$ and $\QB\text{-}\CAT$ are 2-equivalent.
\end{prop}

In what follows, for a concrete category $\CE$ over $\CB$, we write $\overline{\CE}$ for the corresponding $\QB$-category. Hence, $\ob\overline{\CE}=\ob\CE$ and
$$\overline{\CE}(X,Y):=|\CE(X,Y)|\subseteq\CB(|X|,|Y|)$$
for all $X,Y\in\ob\CE$.

Let us also fix the notation for the right adjoints of the join preserving functions
$$-\circ u:\CQ(T,U)\to\CQ(S,U)\quad\text{and}\quad v\circ -:\CQ(S,T)\to\CQ(S,U)$$
for $u:S\to T$ and $v:T\to U$, respectively, in any quantaloid $\CQ$. They are defined such that the equivalences
$$v\leq w\lda u\iff v\circ u\leq w\iff u\leq v\rda w$$
hold for all $u$, $v$ as above, and $w:S\to U$ in $\CQ$, i.e.,
\begin{align*}
& w\lda u=\bv\{ v\in\CQ(T,U)\mid v\circ u\leq w\}\quad\text{and}\\
& v\rda w=\bv\{ u\in\CQ(S,T)\mid v\circ u\leq w\}.
\end{align*}

For $\CQ=\QB$ and $\Bf\subseteq\CB(S,T)$, $\Bg\subseteq\CB(T,U)$ and $\Bh\subseteq\CB(S,U)$, these formulas give
\begin{align*}
&\Bh\lda\Bf=\{g\in\CB(T,U)\mid\forall f\in\Bf:\ g\circ f\in\Bh\}\quad\text{and}\\
&\Bg\rda\Bh=\{f\in\CB(S,T)\mid\forall g\in\Bg:\ g\circ f\in\Bh\}.
\end{align*}

\section{Structured sinks as presheaves, topological functors as total categories} \label{sink_topological}

For a category $\CE$ concrete over $\CB$ and an object $T$ in $\CB$, a \emph{structured sink} $\si$ over $T$ is given by a (possibly large) family of objects $X_i$ in $\CE$ and maps $f_i:|X_i|\to T$, $i\in I$. A \emph{lifting} of $\si=(T,X_i,f_i)_{i\in I}$ is an $\CE$-object $Y$ with $|Y|=T$ such that all maps $f_i$ are $\CE$-morphisms, and the lifting is \emph{final} (w.r.t. $|\text{-}|:\CE\to\CB$) if any map $g:|Y|\to|Z|$ becomes an $\CE$-morphism as soon as all maps $g\circ f_i:|X_i|\to|Z|$ are $\CE$-morphisms. The finality property means precisely
\begin{align*}
\overline{\CE}(Y,Z)&=\{g\in\CB(|Y|,|Z|)\mid\forall i\in I:\ g\circ f_i\in\overline{\CE}(X_i,Z)\}\\
&=\bigcap_{i\in I}\overline{\CE}(X_i,Z)\lda\{f_i\}
\end{align*}
for all $Z\in\ob\CE$.

Replacing the singleton sets $\{f_i\}$ by subsets $\Bf_i\subseteq\CB(|X_i|,T)$, we may assume that the objects $X_i$ are pairwise distinct. In fact, since $\Bf_i$ is allowed to be empty, without loss of generality, we may always assume the indexing class of $\si$ to be $\ob\CE$. So, our structured sink $\si$ is now described as an $\ob\CE$-indexed family of subsets $\si_X\subseteq\CB(|X|,T)$, and its final lifting $Y$ is characterized by
\begin{equation} \label{si_sup} \tag{$\ast$}
\overline{\CE}(Y,Z)=\bigcap_{X\in\ob\CE}\overline{\CE}(X,Z)\lda\si_X
\end{equation}
for all $Z\in\ob\CE$.

Moreover, since for all $X,Z\in\ob\CE$, one trivially has
$$\overline{\CE}(X,Z)\lda\si_X=\bigcap_{X'\in\ob\CE}\overline{\CE}(X',Z)\lda(\si_X\circ\overline{\CE}(X',X)),$$
we may assume $\si$ to be closed under composition with $\CE$-morphisms from the right. That is,
\begin{equation} \label{si_dist} \tag{$\ast\ast$}
\si_X\circ\overline{\CE}(X',X)\subseteq\si_{X'}
\end{equation}
for all $X,X'\in\ob\CE$.

Recall that a faithful functor $|\text{-}|:\CE\to\CB$ is \emph{topological} if all structured sinks admit final liftings. The above considerations show:
\begin{prop}
A concrete category $\CE$ is topological over $\CB$ if, for all families $\si=(T,\si_X)_{X\in\ob\CE}$ with $\si_X\subseteq\CB(|X|,T)$ satisfying (\ref{si_dist}), there is $Y\in\ob\CE$ satisfying (\ref{si_sup}).
\end{prop}

In the terminology of quantaloid-enriched categories (see \cite{Shen2014,Shen2013a,Stubbe2005}), topologicity is therefore characterized by the existence of suprema of all presheaves. Indeed, for any quantaloid $\CQ$, a \emph{presheaf} $\varphi$ on a $\CQ$-category $\CE$ of \emph{extent} $T$ is given by a family of arrows
$$\varphi_X:|X|\ra T$$
in $\CQ$ with $\varphi_X\circ\CE(X',X)\leq\varphi_{X'}$ for all $X,X'\in\ob\CE$. A \emph{supremum} of $\varphi$ is an object $Y=\sup\varphi$ in $\CE$ with extent $T$ satisfying (\ref{si_sup}) transformed to the current context:
\begin{equation} \label{phi_sup} \tag{$\ast'$}
\CE(\sup\varphi,Z)=\bw_{X\in\ob\CE}\CE(X,Z)\lda\varphi_X
\end{equation}
for all $Z\in\ob\CE$.

This latter condition is expressed more compactly once the presheaves on $\CE$ have been organized as a (very large) $\CQ$-category $\PE$, with hom-arrows
$$\PE(\varphi,\psi)=\bw_{X\in\ob\CE}\psi_X\lda\varphi_X.$$
In fact, choosing for $\psi$ the presheaf $\CE(-,Z)$ of extent $|Z|$, condition (\ref{phi_sup}) reads as
\begin{equation} \label{phi_sup_PE} \tag{$\ast''$}
\CE(\sup\varphi,Z)=\PE(\varphi,\CE(-,Z))
\end{equation}
for all $Z\in\ob\CE$. In terms of the \emph{Yoneda} $\CQ$-functor
$$\sY_{\CE}:\CE\to\PE,\quad Z\mapsto\CE(-,Z),$$
this condition reads as
\begin{equation} \label{sup_Y_adjoint} \tag{$\ast'''$}
\CE(\sup\varphi,Z)=\PE(\varphi,\sY Z)
\end{equation}
for all $Z\in\ob\CE$, $\varphi\in\ob\PE$. In other words, with the usual adjointness terminology transferred to the $\CQ$-context, one obtains:

\begin{thm}[Garner \cite{Garner2014}] \label{topological_total}
A concrete category $\CE$ over $\CB$ is topological if, and only if, the $\QB$-category $\overline{\CE}$ is total; that is, if the Yoneda $\QB$-functor $\sY_{\overline{\CE}}$ has a left adjoint.
\end{thm}

For general $\CQ$, a $\CQ$-category $\CE$ is called \emph{total} if the Yoneda $\CQ$-functor $\sY_{\CE}$ has a left adjoint. Total $\CQ$-categories will be studied further beginning from Section \ref{weighted_colimit}.

\section{Cofibred categories as tensored categories} \label{cofibred}

Recall that a faithful functor $|\text{-}|:\CE\to\CB$ is a \emph{cofibration} if every structured singleton-sink has a final lifting; that is, if for every map $f:|X|\to T$ with $X\in\ob\CE$ there is some $Y\in\ob\CE$ with
$$\overline{\CE}(Y,Z)=\overline{\CE}(X,Z)\lda\{f\}$$
for all $Z\in\ob\CE$. We write $Y=f\star X$ and call $\CE$ \emph{cofibred} (over $\CB$) in this case. Dually, $\CE$ is \emph{fibred} (over $\CB$) if $|\text{-}|^{\op}:\CE^{\op}\to\CB^{\op}$ is a cofibration.

Denoting by $\CE_T$ the \emph{fibre} of $|\text{-}|$ at $T\in\ob\CB$, which is a possibly large ordered class, one has the following well-known characterization of topologicity of concrete categories (see \cite{Tholen1979}, \cite[Theorem II.5.9.1]{Hofmann2014}):

\begin{thm} \label{topological_fibre_complete}
A concrete category $\CE$ over $\CB$ is topological if, and only if, it is fibred and cofibred and its fibres are large-complete ordered classes.
\end{thm}

A proof of this theorem appears below as Corollary \ref{topological_dual}.

Wyler \cite{Wyler1971a} originally introduced topological categories using this characterization, which implies in particular the self-duality of topologicity, but restricting himself to concrete categories with \emph{small} fibres. Without that restriction, self-duality (so that $\CE$ is topological over $\CB$ if and only if $\CE^{\op}$ is topological over $\CB^{\op}$) was first established by Hoffmann \cite{Hoffmann1972,Hoffmann1976}.

Although Theorem \ref{topological_fibre_complete} has a natural generalization to the context of quantaloid-enriched categories (which will be discussed in the next section), there is no immediate ``translation'' of the notion of (co)fibration into that context. If, however, we require every $\QB$-arrow $\Bf:|X|\to T$ with $X\in\ob\CE$ (instead of just a single map $f$ in $\CB$) to have a final lifting $Y\in\ob\CE$, so that
\begin{equation} \label{Bf_tensor} \tag{$\dag$}
\overline{\CE}(Y,Z)=\overline{\CE}(X,Z)\lda\Bf
\end{equation}
for all $Z\in\ob\CE$, then we do obtain a notion which is familiar for quantaloid-enriched categories, as we recall next.

For any quantaloid $\CQ$, a $\CQ$-category $\CE$ is \emph{tensored} \cite{Shen2013a,Stubbe2006} if for all $X\in\ob\CE$ and $u:|X|\to T$ in $\CQ$, there is some $Y\in\ob\CE$ with $|Y|=T$ and
\begin{equation} \label{f_tensor} \tag{$\dag'$}
\CE(Y,Z)=\CE(X,Z)\lda u
\end{equation}
for all $Z\in\ob\CE$. Such object $Y$, also called the \emph{tensor} of $u$ and $X$, we denote by $u\star X$. The terminology and notation deserves some explanation.

First of all, for every object $S$ in $\CQ$, one has the $\CQ$-category $\{S\}$ with only one object $S$ with extent $S$, and hom-arrow $\{S\}(S,S)=1_S$. The presheaf $\CQ$-category $\sP\{S\}$ has arrows $u:S\to T$ in $\CQ$ as its objects, and their extent is their codomain: $| u|=T$; the hom-arrows are given by
$$\sP\{S\}( u, v)= v\lda u=\bv\{ w\in\CQ(T,U)\mid w\circ u\leq v\}$$
where $v:S\to U$.
Now, for every object $X$ of a $\CQ$-category $\CE$, there is a $\CQ$-functor
$$\CE(X,-):\CE\to\sP\{S\}$$
with $S=|X|$, and we can rewrite (\ref{f_tensor}) as
\begin{equation} \tag{$\dag''$}
\CE( u\star X,Z)=\sP\{S\}( u,\CE(X,Z))
\end{equation}
for all $Z\in\ob\CE$. This shows instantly:

\begin{prop}[Stubbe \cite{Stubbe2006}]
A $\CQ$-category $\CE$ is tensored if, and only if, for all objects $X$ in $\CE$, the $\CQ$-functor $\CE(X,-):\CE\to\sP\{|X|\}$ has a left adjoint (given by tensor).
\end{prop}

As explained above, for the $\QB$-category $\overline{\CE}$ of a concrete category $\CE$ over $\CB$ to be tensored, it is necessary that $\CE$ be cofibred over $\CB$. Here is the precise differentiation of the two properties in question:

\begin{prop} \label{tensor_cofibration}
Let $\CE$ be a category concrete over $\CB$. Then the $\QB$-category $\overline{\CE}$ is tensored if, and only if, $\CE$ is cofibred with the additional property that for all $X\in\ob\CE$, $T\in\ob\CB$ and $\Bf\subseteq\CB(|X|,T)$, there is $Y\in\ob\CE$ with $|Y|=T$ and
\begin{equation} \label{tensor_cofibre} \tag{$\ddagger$}
\overline{\CE}(Y,Z)=\bigcap_{f\in\Bf}\overline{\CE}(f\star X,Z)
\end{equation}
for all $Z\in\ob\CE$.
\end{prop}

\begin{proof}
It suffices to show that the right-hand sides of (\ref{Bf_tensor}) and (\ref{tensor_cofibre}) coincide. But for all $g\in\CB(T,|Z|)$, one has
\begin{align*}
g\in\overline{\CE}(X,Z)\lda\Bf&\iff\forall f\in\Bf:\ g\circ f\ \text{is an}\ \CE\text{-morphism}\ X\to Z\\
&\iff\forall f\in\Bf:\ g\ \text{is an}\ \CE\text{-morphism}\ f\star X\to Z\\
&\iff g\in\bigcap_{f\in\Bf}\overline{\CE}(f\star X,Z).
\end{align*}
\end{proof}

\begin{cor}
For a non-empty category $\CE$ concrete over $\CB$ with $\overline{\CE}$ tensored, the functor $|\text{-}|:\CE\to\CB$ has a fully faithful left adjoint and right inverse functor.
\end{cor}

\begin{proof}
Independently of the chosen $X$ in $\CE$, for $\Bf=\varnothing$ condition (\ref{tensor_cofibre}) reads as
$$\overline{\CE}(Y,Z)=\CB(T,|Z|)$$
for all $Z\in\ob\CE$, exhibiting $Y$ as the value of the left adjoint at $T$.
\end{proof}

\begin{rem} \label{cofibre_not_tensor}
\emph{For a cofibred category $\CE$ over $\CB$, the $\QB$-category $\overline{\CE}$ may fail to be tensored; in fact, $|\text{-}|:\CE\to\CB$ may fail to have a fully faithful left adjoint.} Indeed, for any category $\CB$ that is not just an ordered class (so that there is at least one hom-set $\CB(T,Z)$ with at least two morphisms), the codomain functor
$$(T\da\CB)\to\CB$$
of the comma category of $\CB$ under $T$ is cofibred over $\CB$ but does not have a fully faithful left adjoint (at $T$). Otherwise there would be a map $j:T\to T$ such that $1_T:T\to\cod j$ would serve as the adjunction unit at $T$, i.e.,
$$\CB(T,\cod g)=(T\da\CB)(j,g)$$
for all $g:T\to Z$. Considering first $g=1_T$ one sees that one must have also $j=1_T$, and then that $\CB(T,Z)$ can contain only at most one map.
\end{rem}

\begin{rem} \label{complete_cofibred_tensored}
In the order of the fibre $\CE_T$ of the functor $|\text{-}|:\CE\to\CB$ over $T\in\ob\CB$, an object $Y$ satisfying (\ref{tensor_cofibre}) is necessarily the join of the objects $(f\star X)_{f\in\Bf}$: just constrain the objects $Z$ to range in $\CE_T$. Hence, when $\overline{\CE}$ is tensored, the fibres of $|\text{-}|$ admit certain small-indexed joins.

However, \emph{for a cofibred category $\CE$ over $\CB$ whose fibres are complete, $\overline{\CE}$ may still fail to be tensored,} as is shown by the following example.
\end{rem}

\begin{exmp} \label{cofibred_not_tensored_exmp} \label{fibre_note_tensored_example}
Consider the category $\CB$ of complete lattices, with monotone (but not necessarily join-preserving) functions as maps; and let $\CE$ be the category of pointed complete lattices, with morphisms $f:(X,x_0)\to(Y,y_0)$ monotone functions satisfying $f(x_0)\leq y_0$. (Hence, $\CE$ is a lax version of the comma category $(1\da\CB)$.) Then $\CE$ is cofibred over $\CB$, and the fibre of the forgetful functor $\CE\to\CB$ at any $X$ is isomorphic to the complete lattice $X$ itself. But considering any monotone function $g:Y\to Z$ in $\CB$ that fails to preserve joins, so that there are $y_i\in Y$, $i\in I$ with
$$z=\bv_{i\in I}g(y_i)<g(y),\quad y=\bv_{i\in I}y_i,$$
for the constant maps $f_i:1\to Y$ with value $y_i$ we then have $(Y,y_i)=f_i\star 1$ and $g\in\overline{\CE}(f_i\star 1,(Z,z))$ for all $i\in I$, but $g\not\in\overline{\CE}((Y,y),(Z,z))$, even though $(Y,y)=\displaystyle\bv\limits_{i\in I}(Y,y_i)$ in the fibre $\CE_Y$. Hence, condition (\ref{tensor_cofibre}) of Proposition \ref{tensor_cofibration} is violated.
\end{exmp}

These remarks underline the difference between order-completeness and conical cocompleteness in quantaloid-enriched categories. Indeed, for a general quantaloid $\CQ$, a $\CQ$-category $\CE$ is \emph{order-complete} \cite{Stubbe2006} if, for all $T\in\ob\CQ$, the class $\CE_T$ of $\CE$-objects of extent $T$ ordered by
$$Y\leq Y'\iff 1_T\leq\CE(Y,Y')$$
admits all joins (and equivalently, meets). In the presheaf $\CQ$-category $\PE$, this order amounts to the componentwise order inherited from $\CQ$:
\begin{align*}
\varphi\leq\psi&\iff 1_T\leq\PE(\varphi,\psi)\\
&\iff\forall X\in\ob\CE:\ 1_T\leq\psi_X\lda\varphi_X\\
&\iff\forall X\in\ob\CE:\ \varphi_X\leq\psi_X.
\end{align*}
It is therefore clear that $\PE$ is order-complete.

One calls $\CE$ \emph{conically cocomplete} \cite{Stubbe2006} if, for all $T\in\ob\CQ$, joins of representable presheaves taken in $(\PE)_T$ have suprema, that is: if for any (possibly large) family of objects $Y_i\in\CE_T$, $i\in I$, the supremum $Y$ of the presheaf
$$\varphi=\bv_{i\in I}\CE(-,Y_i)$$
exists. Such $Y$ must necessarily satisfy $Y=\displaystyle\bv\limits_{i\in I}Y_i$ in $\CE_T$. Indeed, the restriction of condition (\ref{phi_sup}) (see Section \ref{sink_topological}) yields
\begin{align*}
Y\leq Z&\iff\forall X\in\ob\CE:\ 1_T\leq\CE(X,Z)\lda\varphi_X\\
&\iff\forall X\in\ob\CE:\ \varphi_X\leq\CE(X,Z)\\
&\iff\forall X\in\ob\CE,\forall i\in I:\ \CE(X,Y_i)\leq\CE(X,Z)\\
&\iff\forall i\in I:\ Y_i\leq Z.
\end{align*}

\begin{cor} \label{tensor_cofibred_conical}
If the concrete category $\CE$ over $\CB$ is cofibred with $\overline{\CE}$ conically cocomplete, then $\overline{\CE}$ is tensored.
\end{cor}

\begin{proof}
We show that for all $X\in\ob\CE$, $T\in\ob\CB$ and $\Bf\subseteq\CB(|X|,T)$, an object $Y$ in $\CE$ with $|Y|=T$ satisfying (\ref{tensor_cofibre}) in Proposition \ref{tensor_cofibration} is exactly the conical colimit of the objects $f\star X$, $f\in\Bf$. It suffices to show that the right-hand sides of (\ref{tensor_cofibre}) and (\ref{si_sup}) in Section \ref{sink_topological} coincide for $\si=\displaystyle\bigcup\limits_{f\in\Bf}\overline{\CE}(-,f\star X)$. Indeed, for all $Z\in\ob\CE$, $g:T\to Z$ in $\CB$,
\begin{align*}
g\in\bigcap_{f\in\Bf}\overline{\CE}(f\star X,Z)\iff{}&\forall f\in\Bf:\ g\ \text{is an}\ \CE\text{-morphism}\ f\star X\to Z\\
\iff{}&\forall f\in\Bf,W\in\ob\CE,h\in\overline{\CE}(W,f\star X):\\
{}&g\circ h\ \text{is an}\ \CE\text{-morphism}\ W\to Z\\
\iff{}&\forall W\in\ob\CE,h\in\bigcup_{f\in\Bf}\overline{\CE}(W,f\star X):\\
{}&g\circ h\ \text{is an}\ \CE\text{-morphism}\ W\to Z\\
\iff{}& g\in\bigcap_{W\in\ob\CE}\Big(\overline{\CE}(W,Z)\lda\bigcup\limits_{f\in\Bf}\overline{\CE}(W,f\star X)\Big).
\end{align*}
\end{proof}

However, for a concrete category $\CE$ over $\CB$ such that $\overline{\CE}$ is tensored and order-complete, $\overline{\CE}$ may still fail to be conically cocomplete, as is shown by the following example:

\begin{exmp} \label{tensor_not_cocomplete}
Let the category $\CB$ have only one object with exactly two endomorphisms $i,e$, the non-identity morphism $e$ being idempotent. The objects of the category $\CE$ are the natural numbers, plus a largest element $\infty$ adjoined, and its hom-sets are given by
$$\CE(X,Y)=\begin{cases}
\{i,e\}, & \text{if}\ X\leq Y,\\
\{e\}, & \text{if}\ Y<X<\infty,\\
\varnothing, & \text{if}\ X=\infty,Y<\infty.
\end{cases}$$
With composition as in $\CB$ and $|\text{-}|:\CE\to\CB$ mapping morphisms identically, $\CE$ is concrete over $\CB$. Furthermore, $\overline{\CE}$ is tensored since, for all $X\in\ob\CE$, $\overline{\CE}(X,-)$ preserves meets (thus has a left adjoint). But $\overline{\CE}$ is not conically cocomplete since $\varphi=\displaystyle\bv\limits_{X<\infty}\overline{\CE}(-,X)$ does not have a supremum. Indeed, the supremum would have to be the join $\infty$ of the objects $X<\infty$; but for any $Z<\infty$ one has
$$\bigcap_{Y\in\ob\CE}\overline{\CE}(Y,Z)\lda\varphi_Y=\{e\}$$
while $\overline{\CE}(\infty,Z)=\varnothing$.
\end{exmp}

What is needed to make order-completeness equivalent to conical cocompleteness is fibredness:

\begin{prop} \label{fibre_order_complete}
For a fibred category $\CE$ over $\CB$, the $\QB$-category $\overline{\CE}$ is conically cocomplete whenever it is order-complete.
\end{prop}

\begin{proof}
Let $Y=\displaystyle\bv\limits_{i\in I}Y_i$ in $\CE_T$, for $T\in\ob\CB$. With $\varphi=\displaystyle\bv\limits_{i\in I}\CE(-,Y_i)$ we must show
$$\overline{\CE}(Y,Z)=\bigcap_{X\in\ob\CE}\overline{\CE}(X,Z)\lda\varphi_X$$
for all $Z\in\ob\CE$. But for a map $g:T\to|Z|$ in the right-hand side set and $\widetilde{Y}:=(g\rat Z)$ the $(|\text{-}|)$-initial lifting of $g$ one has $Y_i\leq\widetilde{Y}$ for every $i\in I$ (consider $X=Y_i$). Consequently, $Y\leq\widetilde{Y}$, and therefore $g:|Y|\to|Z|$ must be an $\CE$-morphism. This proves ``$\supseteq$'', the other inclusion being trivial.
\end{proof}

\begin{rem}
We note, however, that a conically cocomplete concrete category need not be fibred. Indeed, similarly to Example \ref{fibre_note_tensored_example} consider for $\CB$ the category of complete lattices with join-preserving maps, and let $\CE$ be pointed complete lattices with morphisms preserving joins strictly but base-points only laxly. Then $\CE$ is obviously not fibred over $\CB$ although $\overline{\CE}$ is conically cocomplete.
\end{rem}

\section{Distributors, weighted colimits, total cocompleteness} \label{weighted_colimit}

Suprema of presheaves (as used in Section \ref{sink_topological}) are special weighted colimits that we should mention here in full generality. For this, in turn, it is convenient to have the language of distributors at one's disposal. For a quantaloid $\CQ$, a \emph{$\CQ$-distributor} (also \emph{bimodule} or \emph{profunctor}) $\Phi:\CE\oto\CD$ of $\CQ$-categories $\CE,\CD$ is given by a family of arrows $\Phi(X,Y):|X|\to|Y|$ in $\CQ$ ($X\in\ob\CE,Y\in\ob\CD$) in $\CQ$ such that
$$\CD(Y,Y')\circ\Phi(X,Y)\circ\CE(X',X)\leq\Phi(X',Y')$$
($X,X'\in\ob\CE,Y,Y'\in\ob\CD$). Every $\CQ$-category $\CE$ may be considered as a $\CQ$-distributor $\CE:\CE\oto\CE$ and, in fact, serves as an identity $\CQ$-distributor when one defines the composite of $\Phi$ followed by $\Psi:\CD\oto\CC$ via
$$(\Psi\circ\Phi)(X,Z)=\bv_{Y\in\ob\CD}\Psi(Y,Z)\circ\Phi(X,Y).$$
With the pointwise order inherited from $\CQ$, we obtain the 2-(meta)category
$$\CQ\text{-}\DIS$$
of $\CQ$-categories and their $\CQ$-distributors. $\CQ\text{-}\DIS$ is in fact a (very large) quantaloid, i.e., enriched over $\SUP$, the (meta)category of large-complete ordered classes and sup-preserving functions.

Every $\CQ$-functor $F:\CE\to\CD$ gives rise to the $\CQ$-distributors
\begin{align*}
F_{\nat}:\CE\oto\CD,\quad& F_{\nat}(X,Y)=\CD(FX,Y),\\
F^{\nat}:\CD\oto\CE,\quad& F^{\nat}(Y,X)=\CD(Y,FX),
\end{align*}
so that one has 2-functors
$$(-)_{\nat}:(\QCAT)^{\co}\to\CQ\text{-}\DIS,\quad (-)^{\nat}:(\QCAT)^{\op}\to\CQ\text{-}\DIS,$$
which map objects identically. Here ``$\co$'' refers to the dualization of 2-cells; while $(-)_{\nat}$ is covariant on 1-cells but inverts their order, $(-)^{\nat}$ is contravariant on 1-cells but keeps their order:
\begin{align*}
F\leq F'&\iff\forall X\in\ob\CE:\ 1_{|X|}\in\CD(FX,F'X)\\
&\iff\forall X,Y\in\ob\CE:\ \CD(FX,Y)\geq\CD(F'X,Y)\iff F_{\nat}\geq (F')_{\nat}\\
&\iff\forall X,Y\in\ob\CE:\ \CD(Y,FX)\leq\CD(Y,F'X)\iff F^{\nat}\leq (F')^{\nat}.
\end{align*}

Since $\CE\leq F^{\nat}\circ F_{\nat}$ and $F_{\nat}\circ F^{\nat}\leq\CD$, one has $F_{\nat}\dv F^{\nat}$ in $\CQ\text{-}\DIS$, an important fact that we exploit next.

For a $\CQ$-category $\CE$, every presheaf $\varphi\in\PE$ may be considered as a $\CQ$-distributor $\varphi:\CE\oto\{|\varphi|\}$ when one writes $\varphi(X,|\varphi|)$ for $\varphi_X$. Every $\CQ$-functor $F:\CE\to\CD$ now gives the $\CQ$-functor
$$F^*:\PD\to\PE,\quad\psi\mapsto\psi\circ F_{\nat},$$
so that
$$(F^*\psi)_X=(\psi\circ F_{\nat})(X,|\psi|)=\psi(FX,|\psi|)=\psi_{FX}\quad (X\in\ob\CE).$$
The following lemma is well known:

\begin{lem} \label{Fshriek_Fstar_adjoint}
$F^*$ has a left adjoint $F_!$, given by
$$(F_!\varphi)_Y=\bv_{X\in\ob\CE}\varphi_X\circ\CD(Y,FX).$$
\end{lem}

\begin{proof}
The given formula translates to
$$F_!\varphi=\varphi\circ F^{\nat}.$$
Consequently, $F_!\dv F^*$ follows easily from $F_{\nat}\dv F^{\nat}$.
\end{proof}

Let $D:\CJ\to\CE$ be a $\CQ$-functor (considered as a ``diagram'' in $\CE$). For $\varphi\in\PJ$, a \emph{weighted colimit of $D$ by $\varphi$} is an object $Y$ in $\CE$ with $|Y|=|\varphi|$ and
$$\CE(Y,Z)=\PJ(\varphi,\CE(D-,Z))$$
for all $Z\in\ob\CE$; one writes $Y=\varphi\star D$ in this case. Here $\CE(D-,Z)$ is the value of the composite $\CQ$-functors
$$\CE\to^{\sY_{\CE}}\PE\to^{D^*}\PJ$$
at $Z$.

Since $D_!\dv D^*$, so that $\PE(D_!\varphi,\CE(-,Z))=\PJ(\varphi,D^*\sY_{\CE}Z)$, the weighted colimit $\varphi\star D$ exists precisely when $\sup D_!\varphi$ exists, and then
$$\varphi\star D\cong{\sup}D_!\varphi.$$

\begin{rem} \label{sup_tensor_colimit}
The supremum of $\varphi\in\ob\PE$ is precisely the weighted colimit of the identity $\CQ$-functor of $\CE$ by $\varphi$: $\sup\varphi=\varphi\star 1_{\CE}$. The tensor of $X\in\ob\CE$ and $u:|X|\to T$ in $\CQ$ is precisely the weighted colimit of $\{|X|\}\to\CE$, $|X|\mapsto X$, by $u$: $u\star X=u\star(\{|X|\}\to\CE)$.
\end{rem}

A $\CQ$-category $\CE$ is \emph{totally cocomplete} if $\varphi\star D$ exists for any diagram $D$ in $\CE$ and weight $\varphi$; equivalently, if for every $\CQ$-functor $D$ with codomain $\CE$, the composite $\CQ$-functor $D^*\sY_{\CE}$ has a left adjoint. This is certainly the case when $\CE$ is total, so that $\sY_{\CE}$ has a left adjoint, since $D^*$ has always a left adjoint, by Lemma \ref{Fshriek_Fstar_adjoint}. More comprehensively, we may now state:

\begin{thm}[Stubbe \cite{Stubbe2005}] \label{total_equivalent}
The following are equivalent for a $\CQ$-category $\CE$:
\begin{itemize}
\item[\rm (i)] $\CE$ is total;
\item[\rm (ii)] $\CE$ is totally cocomplete;
\item[\rm (iii)] $\CE$ has all suprema;
\item[\rm (iv)] $\CE$ is tensored and conically cocomplete.
\end{itemize}
\end{thm}

\begin{proof}
(i)${}\Lra{}$(ii): See above.

(ii)${}\Lra{}$(iii) \& (iv): By Remark \ref{sup_tensor_colimit}.

(iii)${}\Lra{}$(i): By definition of totality.

(iv)${}\Lra{}$(iii): Given $\varphi\in\ob\PE$, since $\CE$ is tensored, for every $X\in\ob\CE$ one has $Y_X=\varphi_X\star X$ with $|Y_X|=|\varphi|$ and, since $\CE$ is conically cocomplete, there is $Y=\sup\psi$ with
$$\psi=\bv_{X\in\ob\CE}\CE(-,Y_X).$$
Now, for all $Z\in\ob\CE$, the following calculation is easily validated:
\begin{align*}
\CE(Y,Z)&=\bw_{W\in\ob\CE}\CE(W,Z)\lda\psi_W\\
&=\bw_{W\in\ob\CE}\bw_{X\in\ob\CE}\CE(W,Z)\lda\CE(W,Y_X)\\
&=\bw_{X\in\ob\CE}\bw_{W\in\ob\CE}\CE(W,Z)\lda\CE(W,Y_X)\\
&=\bw_{X\in\ob\CE}\CE(Y_X,Z)\\
&=\bw_{X\in\ob\CE}\CE(X,Z)\lda\varphi_X.
\end{align*}
Consequently, $Y=\sup\varphi$.
\end{proof}

\begin{cor}
For a concrete category $\CE$ over $\CB$, the following are equivalent:
\begin{itemize}
\item[\rm (i)] $\CE$ is topological over $\CB$;
\item[\rm (ii)] the $\QB$-category $\overline{\CE}$ is tensored and conically cocomplete;
\item[\rm (iii)] $\CE$ is cofibred over $\CB$, and $\overline{\CE}$ is conically cocomplete.
\end{itemize}
\end{cor}

\begin{proof}
(i)$\iff$(ii) follows from Theorems \ref{topological_total} and \ref{total_equivalent}, and (ii)$\iff$(iii) follows from Corollary \ref{tensor_cofibred_conical}.
\end{proof}

\section{Dualization} \label{dualization}

Let us now show how the self-duality of topologicity (as stated in Theorem \ref{topological_fibre_complete}) plays itself out for general $\CQ$-categories. First of all, for any quantaloid $\CQ$, dualized as an ordinary category, $\CQ^{\op}$ is a quantaloid again, with $\CQ^{\op}(T,S)=\CQ(S,T)$ carrying the same order. Every $\CQ$-category $\CE$ induces the $\CQ^{\op}$-category $\CE^{\op}$ with $\CE^{\op}(Y,X)=\CE(X,Y)$, and a $\CQ$-functor $F:\CE\to\CD$ becomes a $\CQ^{\op}$-functor $F^{\op}:\CE^{\op}\to\CD^{\op}$. But when $F\leq F'$ for $F':\CE\to\CD$, one has $(F')^{\op}\leq F^{\op}$. Briefly, there is a 2-isomorphism
$$(-)^{\op}:(\QCAT)^{\co}\to\CQ^{\op}\text{-}\CAT.$$
One can now dualize the constructions and notions encountered thus far, as follows:
\begin{itemize}
\item $\PdE:=(\sP(\CE^{\op}))^{\op}$ (the \emph{covariant presheaf category} of $\CE$, as opposed to the \emph{contravariant} presheaf category $\PE$);
\item $\inf_{\CE}\varphi:=\sup_{\CE^{\op}}\varphi$ (the \emph{infimum} of $\varphi\in\PdE$);
\item $\varphi\rat D:=\varphi\star D^{\op}$ (the \emph{weighted limit of $D:\CJ\to\CE$ by $\varphi\in\PdJ$});
\item $\sYd_{\CE}:=(\sY_{\CE^{\op}})^{\op}: \CE\to\PdE$, $X\mapsto\CE(X,-)$ (the \emph{dual Yoneda} $\CQ$-functor);
\item $\CE$ \emph{cototal} $:\iff$ $\CE^{\op}$ total $\iff$ $\sYd_{\CE}$ has a right adjoint;
\item $\CE$ \emph{totally complete} $:\iff$ $\CE^{\op}$ totally cocomplete $\iff$ $\CE$ has all weighted limits;
\item $\CE$ \emph{cotensored} $:\iff$ $\CE^{\op}$ tensored;
\item $\CE$ \emph{conically complete} $:\iff$ $\CE^{\op}$ conically cocomplete.
\end{itemize}

For our next steps, it is convenient to have a $\CQ$-version of the \emph{Adjoint Functor Theorem} at our disposal, as follows:

\begin{prop} \label{left_adjoint_tensor}
Let $\CE$ be tensored and order-complete. Then a $\CQ$-functor $F:\CE\to\CD$ has a right adjoint if, and only if, $F$ preserves tensors and the restrictions $F_T:\CE_T\to\CD_T$ $(T\in\ob\CQ)$ of $F$ to the fibres preserve arbitrary joins.
\end{prop}

\begin{proof}
If $F$ has a right adjoint $G$, then $F$ preserves all existing weighted colimits, in particular tensors, and since $F\dv G$ implies $F_T\dv G_T$ for all $T\in\ob\CQ$, $F$ preserves also all joins in the fibres. Conversely, assuming preservation of tensors and joins, one first observes that every $F_T$ must have a right adjoint $G_T$ since $\CE$ is order-complete. Putting
$$GY=G_{|Y|}Y,$$
for all $X\in\ob\CE$, $Y\in\ob\CD$ one trivially has
$$\CE(X,GY)\leq\CD(FX,FGY)\leq\CD(FX,Y),$$
and from
\begin{align*}
1_{|Y|}&\leq\CE(X,GY)\lda\CE(X,GY)\\
&=\CE(\CE(X,GY)\star X,GY)&(\CE\ \text{tensored})\\
&=\CD(F(\CE(X,GY)\star X),Y)&(F_{|Y|}\dv G_{|Y|})\\
&=\CD(\CE(X,GY)\star FX,Y)&(F\ \text{preserves tensors})\\
&=\CD(FX,Y)\lda\CE(X,GY)
\end{align*}
one obtains $\CD(FX,Y)\leq\CE(X,GY)$. Hence, $F\dv G$.
\end{proof}

One can now extend the list of equivalent statements of Theorem \ref{total_equivalent} by its dualizations, as follows:

\begin{thm}[Stubbe \cite{Stubbe2005}] \label{total_cototal_equivalent}
A $\CQ$-category $\CE$ is total if and only if the following equivalent conditions hold:
\begin{itemize}
\item[\rm (v)] $\CE$ is cototal;
\item[\rm (vi)] $\CE$ is totally complete;
\item[\rm (vii)] $\CE$ has all infima;
\item[\rm (viii)] $\CE$ is cotensored and conically complete.
\end{itemize}
\end{thm}

\begin{proof}
It suffices to prove that $\CE$ is cototal when $\CE$ is total, and thanks to Proposition \ref{left_adjoint_tensor} and Remark \ref{sup_tensor_colimit}, for that it suffices to prove that the dual Yoneda $\CQ$-functor $\sYd:\CE\to\PdE$ preserves all weighted colimits. Here is a quick sketch of that fact. First note:
\begin{itemize}
\item[\rm (a)] In the (meta)quantaloid $\CQ\text{-}\DIS$, for $\Phi:\CE\oto\CD$, $\Psi:\CD\oto\CC$, $\Xi:\CE\to\CC$ one has
\begin{align*}
(\Xi\lda\Phi)(Y,Z)&=\bw_{X\in\ob\CE}\Xi(X,Z)\lda\Phi(X,Y),\\
(\Psi\rda\Xi)(X,Y)&=\bw_{Z\in\ob\CC}\Psi(Y,Z)\rda\Xi(X,Z)
\end{align*}
for all $X\in\ob\CE$, $Y\in\ob\CD$, $Z\in\ob\CC$.
\item[\rm (b)] For $\Phi:\CE\oto\CC$, $\Psi:\CD\oto\CB$ and a $\CQ$-functor $F:\CE\to\CD$ one has
$$(\Psi\circ F_{\nat})\lda\Phi=\Psi\lda(\Phi\circ F^{\nat}).$$
\item[\rm (c)] $\sY=\sY_{\CE}$ (and, hence, $\sYd_{\CE}$) is fully faithful, that is: $\sY^{\nat}\circ\sY_{\nat}=\CE$.
\end{itemize}

Now consider the weighted colimits $\varphi\star D$ of $D:\CJ\to\CE$ by $\varphi\in\ob\PJ$, then
\begin{align*}
\sYd(\varphi\star D)&=\CE(\varphi\star D,-)&(\text{definition of }\sYd)\\
&=D_{\nat}\lda\varphi&(\text{definition of weighted colimit, (a)})\\
&=((\sYd)^{\nat}\circ\sYd_{\nat}\circ D_{\nat})\lda\varphi&(\text{by (c)})\\
&=((\sYd)^{\nat}\circ(\sYd D)_{\nat})\lda\varphi&(\text{functoriality of }(-)_{\nat})\\
&=(\sYd)^{\nat}\lda(\varphi\circ(\sYd D)^{\nat})&(\text{by (b)})\\
&=\varphi\star\sYd D;
\end{align*}
here the last step follows from the fact, that for any $\CQ$-functor $F:\CJ\to\PdE$ (in lieu of $\sYd D$), the weighted colimit of $F$ by $\varphi$ may be computed as
$$\varphi\star F={\sup}_{\PdE}F_!\varphi=(\sYd)^{\nat}\lda F_!\varphi=(\sYd)^{\nat}\lda(\varphi\circ F^{\nat}).$$
\end{proof}

\begin{cor} \label{topological_dual}
For a concrete category $\CE$ over $\CB$, the following assertions are equivalent:
\begin{itemize}
\item[\rm (i)] $\CE$ is topological over $\CB$;
\item[\rm (ii)] $\CE^{\op}$ is topological over $\CB^{\op}$;
\item[\rm (iii)] $\CE$ is fibred and cofibred over $\CB$ with large-complete fibres.
\end{itemize}
\end{cor}

\begin{proof}
Since $\CQ_{\CB^{\op}}=(\CQ_{\CB})^{\op}$ and $\overline{\CE^{\op}}=\overline{\CE}^{\op}$, the equivalence of (i) and (ii) follows from Theorems \ref{topological_total} and \ref{total_cototal_equivalent}, which also imply (i)$\&$(ii)${}\Lra{}$(iii). The converse implication follows with Corollary \ref{tensor_cofibred_conical} and Proposition \ref{fibre_order_complete}.
\end{proof}

\section{Universality of the presheaf construction} \label{universality}

First, let us briefly recall the \emph{Yoneda Lemma} for $\CQ$-categories:

\begin{lem} \label{Yoneda_lemma}
For a $\CQ$-category $\CE$ and all $X\in\ob\CE$, $\varphi\in\ob\PE$, one has
$$\PE(\sY_{\CE}X,\varphi)=\varphi_X.$$
\end{lem}

As a consequence one obtains the following fundamental adjunction:

\begin{prop} \label{composition_graph_adjunction}
For $\CQ$-categories $\CE$, $\CC$, there is a natural 1-1 correspondence
$$\bfig
\morphism<400,0>[\CE`\CC;\Psi] \place(200,0)[\circ]
\morphism(-150,-100)/-/<700,0>[`;]
\morphism(0,-200)<400,0>[\CC`\PE;G]
\place(1000,-100)[(G^{\nat}\circ(\sY_{\CE})_{\nat}=\Psi),]
\efig$$
which respects the order of $\CQ$-functors and $\CQ$-distributors.
\end{prop}

\begin{proof}
Given $\Psi$, a $\CQ$-functor $G$ with $G^{\nat}\circ(\sY_{\CE})_{\nat}=\Psi$ must necessarily satisfy
$$(GZ)_X=\PE(\sY_{\CE}X,GZ)=\Psi(X,Z)$$
for all $Z\in\ob\CC$, $X\in\ob\CE$. Conversely, defining $G$ in this way one obtains a $\CQ$-functor.
\end{proof}

When restricting ourselves to the case of a small quantaloid $\CQ$ and to considering small $\CQ$-categories, we therefore obtain an adjunction
$$\bfig
\morphism|a|/@{<-}@<6pt>/<700,0>[(\QCat)^{\op}`\CQ\text{-}{\bf Dis}.;\sP]
\morphism|b|/@{->}@<-6pt>/<700,0>[(\QCat)^{\op}`\CQ\text{-}{\bf Dis}.;(-)^{\nat}]
\place(400,0)[\bot]
\efig$$
Here $(-)^{\nat}$ maps objects identically while the presheaf functor $\sP$ assigns to a $\CQ$-distributor $\Phi:\CE\oto\CD$ the $\CQ$-functor
$$\Phi^*:\PD\to\PE$$
with $(\Phi^*)^{\nat}\circ(\sY_{\CE})_{\nat}=(\sY_{\CD})_{\nat}\circ\Phi$, that is:
$$(\Phi^*\psi)_X=\bv_{Z\in\ob\CD}\psi_Z\circ\Phi(X,Z)$$
for all $\psi\in\ob\PD$, $X\in\ob\CE$. The unit of the adjunction at $\CE$ is $(\sY_{\CE})_{\nat}$ while $\sY_{\CE}$ is the counit (since $(\sY_{\CE})^{\nat}\circ(\sY_{\CE})_{\nat}=\CE$). Note that, for a $\CQ$-functor $F:\CE\to\CD$, one has
$$(F_{\nat})^*=F^*\quad\text{and}\quad(F^{\nat})^*=F_!$$
with $F^*$, $F_!$ defined as in Lemma \ref{Fshriek_Fstar_adjoint}. Consequently, there is a monad
$$(\sP,\sY,\sS)$$
on the 2-category $\QCat$, with $\sP$ mapping $F$ to $F_!$ and with the monad multiplication $\sS$ given by
$$\sS_{\CE}\Phi=\bv_{\varphi\in\ob\CE}\Phi_{\varphi}\circ\varphi_X$$
for all $\Phi\in\sP\PE$, $X\in\ob\CE$. It is straightforward to show that this monad is of \emph{Kock-Z{\"o}berlein type}, that is, that
$$(\sY_{\CE})_!=(\sY_{\CE}^{\nat})^*\leq\sY_{\PE}$$
for every $\CQ$-category $\CE$, and consequently
$$\sS_{\CE}\dv\sY_{\PE}:\PE\to\sP\PE.$$
In other words, every $\Phi\in\sP\PE$ has a supremum: $\sup\Phi=\sS_{\CE}\Phi$. This fact of course remains true also for large $\CQ$ and $\CE$; one just has to accept the fact that $\PE$ and $\sP\PE$ will generally live in higher universes than $\CE$.

\begin{cor} \label{Yoneda_fully_faithful}
Yoneda maps every $\CQ$-category $\CE$ fully and faithfully into the total $\CQ$-(meta)category $\PE$.
\end{cor}

\begin{thm}
For all $\CQ$-categories $\CE$, $\CD$ with $\CD$ total, there is a natural 1-1 correspondence
$$\bfig
\morphism<400,0>[\CE`\CD;F]
\morphism(-200,-100)/-/<2050,0>[`;]
\morphism(0,-200)<1000,0>[\PE`\CD,\ H\ \text{preserves weighted colimits};H]
\place(2200,-100)[(H\sY_{\CE}=F)]
\efig$$
which respects the order of $\CQ$-functors.
\end{thm}

\begin{proof}
Let us first note that every $\varphi\in\ob\PE$ is the weighted colimit of $\sY_{\CE}$ by $\varphi$, since
$$\PE(\varphi,\psi)=\PE(\varphi,\PE(\sY_{\CE}-,\psi))$$
for all $\psi\in\ob\PE$. Hence, given $F$, any $H$ with $H\sY_{\CE}=F$ that preserves all weighted colimits must satisfy
$$H\varphi=H(\varphi\star\sY_{\CE})=\varphi\star H\sY_{\CE}=\varphi\star F.$$

Conversely, a straightforward computation shows that $H$ defined in this way is actually the composite $\CQ$-functor
$$\PE\to^{F_!}\PD\to^{\sup_{\CD}}\CD$$
which, as the composite of two left adjoints, must preserve all weighted colimits. Furthermore,
$$H\sY_{\CE}={\sup}_{\CD}F_!\sY_{\CE}={\sup}_{\CD}\sY_{\CD}F=F$$
since $\sY_{\CD}$ is fully faithful.
\end{proof}

Consequently, for small $\CQ$ there is an adjunction
$$\bfig
\morphism|a|/@{<-}@<6pt>/<700,0>[\CQ\text{-}{\bf TotCat}`\QCat,;\sP]
\morphism|b|/@{->}@<-6pt>/<700,0>[\CQ\text{-}{\bf TotCat}`\QCat,;]
\place(400,0)[\bot]
\efig$$
with $\CQ\text{-}{\bf TotCat}$ denoting the category of small total $\CQ$-categories and their weighted-colimit-preserving $\CQ$-functors. The induced monad on $\QCat$ is again $(\sP,\sY,\sS)$, as described before Corollary \ref{Yoneda_fully_faithful}.

\begin{cor}[Herrlich \cite{Herrlich1974}]
For a concrete category over $\CB$, the topological (meta)category $\PE$ over $\CB$ has the following universal property: Every concrete functor $F:\CE\to\CD$ into a topological category $\CD$ over $\CB$ factors uniquely through a concrete functor $H:\PE\to\CD$ with $H\sY_{\CE}=F$ that preserves final sinks.
\end{cor}

Here the objects of $\PE$ are structured sinks satisfying the closure property (\ref{si_dist}) of Section \ref{sink_topological}, and the full concrete embedding $\sY_{\CE}$ assigns to every object $X$ the structured sink of all maps with codomain $|X|$.

\begin{proof}
The only point to observe is that a concrete functor preserves final sinks if and only if it preserves suprema (see Section \ref{sink_topological}) or, equivalently, weighted colimits (see Section \ref{weighted_colimit}).
\end{proof}

As a consequence, the category of small topological categories over the small category $\CB$ and the finality-preserving concrete functors admits a right adjoint forgetful functor into $\CatB$.

\section{Total $\CQ$-categories are induced by Isbell adjunctions} \label{total_Chu}

Let us re-interpret Proposition \ref{composition_graph_adjunction}, by setting up the (meta-)2-category
$$\QCHU$$
whose objects are given by $\CQ$-distributors, and whose morphisms
$$(F,G):\Phi\to\Psi$$
are given by $\CQ$-functors $F:\CE\to\CD$, $G:\CC\to\CB$ which make the diagram
$$\bfig
\square<600,400>[\CE`\CD`\CB`\CC;F_{\nat}`\Phi`\Psi`G^{\nat}]
\place(300,0)[\circ] \place(300,400)[\circ] \place(0,200)[\circ] \place(600,200)[\circ]
\efig$$
commute or, equivalently, satisfy the \emph{diagonal condition}
$$\Psi\lda F^{\nat}=G_{\nat}\rda\Phi.$$
With composition and order defined as in $\QCAT$, $\QCHU$ becomes a (meta-)2-category whose morphisms may also be referred to as \emph{$\CQ$-Chu transforms} (in generalization of the terminology used for morphisms of Chu spaces; see \cite{Barr1991,Pratt1995,Giuli2007}). There is an obvious 2-functor
$$\dom:\QCHU\to\QCAT,\quad (F,G)\mapsto F.$$

\begin{prop}
For all $\CQ$-categories $\CE$ and $\CQ$-distributors $\Psi:\CD\to\CC$, there is a natural 1-1 correspondence
$$\bfig
\morphism<600,0>[\CE`\dom\Psi;F]
\morphism(-200,-100)/-/<1000,0>[`;]
\morphism(0,-250)<600,0>[(\sY_{\CE})_{\nat}`\Psi;(F,G)]
\efig$$
which preserves the order of $\CQ$-functors and $\CQ$-Chu transforms.
\end{prop}

\begin{proof}
Given $F$, by Proposition \ref{composition_graph_adjunction} there is a unique $\CQ$-functor $G:\CC\to\PE$ with $G^{\nat}\circ(\sY_{\CE})_{\nat}=\Psi\circ F_{\nat}$, that is: with $(F,G):(\sY_{\CE})_{\nat}\to\Psi$ a $\CQ$-Chu transform.
\end{proof}

Under the restriction to small $\CQ$-categories (for $\CQ$ small as well) we therefore obtain a full and faithful left adjoint to
$$\dom:\QChu\to\QCat.$$

\begin{cor} \label{QCat_QChu_cofreflective}
The assignment $\CE\mapsto(\sY_{\CE})_{\nat}$ embeds $\QCat$ into $\QChu$ as a full coreflective subcategory.
\end{cor}

\begin{rem}
Note that the embedding of $\QCat\to\QChu$ of Corollary \ref{QCat_QChu_cofreflective} does NOT preserve the local order in $\QCat$, i.e., it is not a 2-categorical embedding.
\end{rem}

Every $\CQ$-distributor $\Phi:\CE\oto\CB$ induces the \emph{Isbell adjunction} \cite{Shen2014,Shen2013a} (in generalization of the terminology introduced by Lawvere \cite{Lawvere1986} for enriched categories)
$$\bfig
\morphism|a|/@{->}@<6pt>/<500,0>[\PE`\PdB,;\uPhi]
\morphism|b|/@{<-}@<-6pt>/<500,0>[\PE`\PdB,;\dPhi]
\place(230,0)[\bot]
\morphism(1000,60)/|->/[\varphi`\Phi\lda\varphi;]
\morphism(900,-60)/<-|/[\psi\rda\Phi`\psi;]
\place(1200,0)[\bot]
\efig.$$
Indeed, the easily established identity
$$\psi\rda(\Phi\lda\varphi)=(\psi\rda\Phi)\lda\varphi$$
translates to
$$\PdB(\uPhi\varphi,\psi)=\PE(\varphi,\dPhi\psi)$$
for all $\varphi\in\ob\PE$, $\psi\in\ob\PB$. We denote by $\sI\Phi$ the (very large) full reflective $\CQ$-subcategory of $\PE$ of presheaves fixed by the adjunction, i.e.,
$$\ob(\sI\Phi)=\{\varphi\in\ob\PE\mid\dPhi\uPhi\varphi=\varphi\},$$
and call $\sI\Phi$ the \emph{Isbell $\CQ$-category} of $\Phi$.

\begin{prop}
$\sI$ may be functorially extended to $\CQ$-Chu transforms such that, for every $\CQ$-functor $F:\CE\to\CD$,
$$\sI((\sY_{\CE})_{\nat}\to^{(F,F^*)}(\sY_{\CD})_{\nat})=F_!:\PE\to\PD.$$
Hence, under the restriction to small objects, one has the commutative diagram:
$$\bfig
\Atriangle/<-`->`->/[\QChu`\QCat`\QCat;(\sY_{\Box})_{\nat}`\sI`\sP]
\efig$$
\end{prop}

\begin{proof}
First, for a $\CQ$-Chu transform $(F,G):\Phi\to\Psi$, one composes $F_!:\PE\to\PD$ with the reflector of $\sI\Phi\ \to/^(->/\PD$ to define
$$\sI(F,G):=\dPsi\uPsi F_!:\sI\Phi\to\sI\Psi.$$
To see the functoriality of $\sI$, we need to check $\sI(H,K)\sI(F,G)=\sI(HF,GK)$, i.e.,
$$\dXi\uXi H_!\dPsi\uPsi F_!=\dXi\uXi H_! F_!$$
where $(H,K):\Psi\to\Xi$. Trivially $\dXi\uXi H_! F_!\leq\dXi\uXi H_!\dPsi\uPsi F_!$ since $1_{\PD}\leq\dPsi\uPsi$ by adjunction. For the reverse inequality, since $\dXi\uXi$ is idempotent, it suffices to prove $H_!\dPsi\uPsi\leq\dXi\uXi H_!$. Indeed, for all $\psi\in\ob\PD$,
\begin{align*}
H_!\dPsi\uPsi\psi&=((\Psi\lda\psi)\rda\Psi)\circ H^{\nat}\\
&\leq((K^{\nat}\circ(\Psi\lda\psi))\rda(K^{\nat}\circ\Psi))\circ H^{\nat}\\
&=((K^{\nat}\circ(\Psi\lda\psi))\rda(\Xi\circ H_{\nat}))\circ H^{\nat}&((H,K)\ \text{is a}\ \CQ\text{-Chu transform})\\
&\leq(K^{\nat}\circ(\Psi\lda\psi))\rda(\Xi\circ H_{\nat}\circ H^{\nat})\\
&\leq(K^{\nat}\circ(\Psi\lda\psi))\rda\Xi&(H_{\nat}\dv H^{\nat})\\
&=((K^{\nat}\circ\Psi)\lda\psi)\rda\Xi&(\text{see explanation below})\\
&=((\Xi\circ H_{\nat})\lda\psi)\rda\Xi&((H,K)\ \text{is a}\ \CQ\text{-Chu transform})\\
&=(\Xi\lda(\psi\circ H^{\nat}))\rda\Xi&(\text{by (b) in the proof of Theorem \ref{total_cototal_equivalent}})\\
&=\dXi\uXi H_!\psi.
\end{align*}
Here the fourth equality from the bottom holds since one easily derives $K^{\nat}\circ\Psi=K_{\nat}\rda\Psi$ for any $\Psi$, and consequently
$$K^{\nat}\circ(\Psi\lda\psi)=K_{\nat}\rda(\Psi\lda\psi)=(K_{\nat}\rda\Psi)\lda\psi=(K^{\nat}\circ\Psi)\lda\psi.$$

Next, for every $\varphi\in\ob\PE$, it is easy to see that
$$(\sY_{\CE})_{\nat}\lda\varphi=\PE(\varphi,-)\in\ob\PdE.$$
Consequently, for all $X\in\ob\CE$, with a repeated application of the Yoneda Lemma one obtains
\begin{align*}
(((\sY_{\CE})_{\nat}\lda\varphi)\rda(\sY_{\CE})_{\nat})_X&=\bw_{\psi\in\ob\PE}\PE(\varphi,\psi)\rda\PE(\sY_{\CE}X,\psi)\\
&=\sPd(\PE)(\PE(\sY_{\CE}X,-),\PE(\varphi,-))\\
&=\PE(\sY_{\CE}X,\varphi)=\varphi_X,
\end{align*}
and therefore $\sI(\sY_{\CE})_{\nat}=\PE$.
\end{proof}

\begin{thm} \label{total_IPhi}
The following assertions on a $\CQ$-category $\CE$ are equivalent:
\begin{itemize}
\item[\rm (i)] $\CE$ is total;
\item[\rm (ii)] $\CE$ is equivalent to a full reflective $\CQ$-subcategory of some presheaf $\CQ$-category;
\item[\rm (iii)] $\CE$ is equivalent to the Isbell $\CQ$-category of some $\CQ$-distributor.
\end{itemize}
\end{thm}

\begin{proof}
(i)${}\Lra{}$(ii): $\CE$ is equivalent to its image under the Yoneda $\CQ$-functor $\sY_{\CE}$, which is reflective in $\PE$ when $\CE$ is total.

(ii)${}\Lra{}$(iii): We may assume that there is a full inclusion $J:\CE\ \to/^(->/\PD$ with left adjoint $L$, for some $\CQ$-category $\CD$. Then one defines a $\CQ$-distributor $\Phi:\CD\oto\CE$ with $\Phi(X,\psi)=\psi_X$ for all $X\in\ob\CD$, $\psi\in\ob\CE\subseteq\ob\PD$. Now it is easy to verify the following calculation for all $\varphi\in\ob\PD$:
\begin{align*}
L\varphi&=\bw_{\psi\in\ob\PD}(L\psi\lda\varphi)\rda L\psi\\
&=\bw_{\psi\in\ob\PD}(\Phi(-,L\psi)\lda\varphi)\rda\Phi(-,L\psi)\\
&=\dPhi\uPhi\varphi.
\end{align*}
Consequently, $JL=\dPhi\uPhi$, from which one derives $\CE=\sI\Phi$, as desired.

(iii)${}\Lra{}$(i): Since an Isbell $\CQ$-category is a full reflective $\CQ$-subcategory of a presheaf $\CQ$-category, which by Corollary \ref{Yoneda_fully_faithful} is total, it suffices to verify that a full reflective $\CQ$-subcategory $\CA$ of a total $\CQ$-category $\CC$ is total. Indeed, if $L\dv J:\CA\ \to/^(->/\CC$, then $L\sup_{\CC}J_!$ serves as a left adjoint to $\sY_{\CA}$.
\end{proof}

\section{Characterization of the Isbell $\CQ$-category of $\CQ$-distributors} \label{IPhi}

Theorem \ref{total_IPhi} shows that every total $\CQ$-category can be seen as the Isbell $\CQ$-category of some $\CQ$-distributor. Conversely, given a $\CQ$-distributor $\Phi$, we now want to characterize its Isbell $\CQ$-category, as follows.

\begin{thm}[Shen-Zhang \cite{Shen2013a}] \label{IPhi_characterization}
Let $\Phi:\CE\oto\CD$ be a $\CQ$-distributor. Then a $\CQ$-category $\CC$ is equivalent to $\sI\Phi$ if, and only if, $\CC$ is total and there are $\CQ$-functors $F:\CE\to\CC$, $G:\CD\to\CC$ with
\begin{itemize}
\item[\rm (1)] $\Phi=G^{\nat}\circ F_{\nat}$;
\item[\rm (2)] $F$ is dense (so that every object $Z$ in $\CC$ is presentable as $Z\cong\varphi\star F$ for some $\varphi\in\ob\PE$);
\item[\rm (3)] $G$ is codense, that is: $G^{\op}$ is dense.
\end{itemize}
\end{thm}

\begin{proof}
The conditions are certainly necessary: given $\Phi$ and assuming $\CC=\sI\Phi\ \to/^(->/\PE$, one defines $F,G$ by
$$FX=\dPhi\uPhi\sY_{\CE}X,\quad GY=\dPhi\sYd_{\CD}Y$$
for all $X\in\ob\CE$, $Y\in\ob\CD$. Then
\begin{align*}
\CC(FX,GY)&=\PE(\dPhi\uPhi\sY_{\CE}X,\dPhi\sYd_{\CD}Y)\\
&=\PdD(\uPhi\sY_{\CE}X,\sYd_{\CD}Y)\\
&=\PE(\sY_{\CE}X,\dPhi\sYd_{\CD}Y)&(\uPhi\dv\dPhi)\\
&=(\dPhi\sYd_{\CD}Y)_X&(\text{Yoneda Lemma})\\
&=\Phi(X,Y),
\end{align*}
so that $G^{\nat}\circ F_{\nat}=\Phi$. Another straightforward calculation shows that every $\varphi\in\ob\CC$ appears as the weighted colimit $\varphi\star F$, so that $F$ is dense; codensity of $G$ follows dually.

For the sufficiency of these conditions, we show that the transpose $\widehat{F_{\nat}}:\CC\to\PE$ of $F_{\nat}$ given by $\widehat{F_{\nat}}Z=F_{\nat}(-,Z)$ for all $Z\in\ob\CC$ is an equivalence of $\CQ$-categories when restricting the codomain to $\sI\Phi$.

First, one notices that $\CC=F_{\nat}\lda F_{\nat}=G^{\nat}\rda G^{\nat}$ whenever $F$ is dense and $G$ is codense. Indeed, by expressing each $Z\in\ob\CC$ as $Z\cong\varphi\star F$ for some $\varphi\in\ob\PE$ one has $\CC(Z,-)=F_{\nat}\lda\varphi$ by the definition of weighted colimits, and consequently
\begin{align*}
\CC(Z,-)&\leq F_{\nat}\lda F_{\nat}(-,Z)\\
&\leq(F_{\nat}\lda F_{\nat}(-,Z))\circ\CC(Z,Z)\\
&=(F_{\nat}\lda F_{\nat}(-,Z))\circ(F_{\nat}(-,Z)\lda\varphi)\\
&\leq F_{\nat}\lda\varphi\\
&=\CC(Z,-),
\end{align*}
from which one derives $\CC=F_{\nat}\lda F_{\nat}$. The assertion $\CC=G^{\nat}\rda G^{\nat}$ follows dually.

Second, the codomain of $\widehat{F_{\nat}}$ may be restricted to $\sI\Phi$ since for all $Z\in\ob\PE$,
\begin{align*}
\dPhi(G^{\nat}(Z,-))&=G^{\nat}(Z,-)\rda\Phi\\
&=G^{\nat}(Z,-)\rda(G^{\nat}\circ F_{\nat})\\
&=(G^{\nat}(Z,-)\rda G^{\nat})\circ F_{\nat}\\
&=\CC(-,Z)\circ F_{\nat}\\
&=F_{\nat}(-,Z).
\end{align*}
Here the third equation holds since one easily derives $\Psi\circ F_{\nat}=\Psi\lda F^{\nat}$ for any $\Psi$, and consequently
$$G^{\nat}(Z,-)\rda(G^{\nat}\circ F_{\nat})=G^{\nat}(Z,-)\rda(G^{\nat}\lda F^{\nat})=(G^{\nat}(Z,-)\rda G^{\nat})\lda F^{\nat}=(G^{\nat}(Z,-)\rda G^{\nat})\circ F_{\nat}.$$

Next, $\widehat{F_{\nat}}$ is fully faithful since for all $X,X'\in\ob\CE$,
$$\sI\Phi(\widehat{F_{\nat}}X,\widehat{F_{\nat}}X')=F_{\nat}(-,X')\lda F_{\nat}(-,X)=\CC(X,X').$$

Finally, to see that $\widehat{F_{\nat}}$ is surjective, for each $\varphi\in\ob\sI\Phi$ one forms $Z=\varphi\star F\in\ob\CC$, then
$$G^{\nat}(Z,-)=G^{\nat}\circ\CC(Z,-)=G^{\nat}\circ(F_{\nat}\lda\varphi)=(G^{\nat}\circ F_{\nat})\lda\varphi=\Phi\lda\varphi=\uPhi\varphi,$$
and consequently $\widehat{F_{\nat}}Z=\dPhi(G^{\nat}(Z,-))=\dPhi\uPhi\varphi=\varphi$.
\end{proof}

\begin{rem} \label{IPhi_fully_faithful}
If the $\CQ$-distributor $\Phi$ in Theorem \ref{IPhi_characterization} satisfies $\Phi\rda\Phi=\CE$, $\Phi\lda\Phi=\CD$, then the functors $F:\CE\to\sI\Phi$, $G:\CD\to\sI\Phi$ as defined in Theorem \ref{IPhi_characterization} may be assumed to be fully faithful. Indeed, for all $X,Y\in\ob\CE$ one then has
\begin{align*}
\sI\Phi(FX,FY)&=\PE(\dPhi\uPhi\sY_{\CE}X,\dPhi\uPhi\sY_{\CE}Y)\\
&=\PdD(\uPhi\sY_{\CE}X,\uPhi\sY_{\CE}Y)\\
&=(\Phi\lda\sY_{\CE}Y)\rda(\Phi\lda\sY_{\CE}X)\\
&=(\Phi\rda\Phi)(X,Y)\\
&=\CE(X,Y),
\end{align*}
and likewise for $G$.
\end{rem}

\begin{cor} \label{dense_codense_subcategory}
Let $\CE$ be a full $\CQ$-subcategory of $\CC$. Then $\CC$ is equivalent to the Isbell $\CQ$-category $\sI\CE$ of (the identity $\CQ$-distributor) $\CE$ if, and only if, $\CC$ is total and $\CE$ is both dense and codense in $\CC$.
\end{cor}

\begin{proof}
The sufficiency of the condition follows with Theorem \ref{IPhi_characterization} when one puts $\Phi=\CE$, with $F=G$ the inclusion $\CQ$-functor to $\CC$. For the necessity one involves Theorem \ref{IPhi_characterization} and Remark \ref{IPhi_fully_faithful}.
\end{proof}

For a concrete category $\CE$ over $\CB$, the Isbell adjunction of $\CE$ (considered as the identity $\QB$-distributor of $\CE$)
$$\bfig
\morphism|a|/@{->}@<6pt>/<500,0>[\PE`\PdE;\CE_{\ua}]
\morphism|b|/@{<-}@<-6pt>/<500,0>[\PE`\PdE;\CE^{\da}]
\place(230,0)[\bot]
\efig$$
is described by
\begin{align*}
(\CE_{\ua}\varphi)_Y&=\bw_{X\in\ob\CE}\CE(X,Y)\lda\varphi_X\\
&=\{g\in\CB(T,|Y|)\mid\forall X\in\ob\CE,f:|X|\to T\ \text{in}\ \varphi_X:\ g\circ f\ \text{is an}\ \CE\text{-morphism}\},\\
(\CE^{\da}\psi)_X&=\bw_{Y\in\ob\CE}\psi_Y\rda\CE(X,Y)\\
&=\{f\in\CB(|X|,T)\mid\forall Y\in\ob\CE,g:T\to|Y|\ \text{in}\ \psi_Y:\ g\circ f\ \text{is an}\ \CE\text{-morphism}\}
\end{align*}
for every structured sink $\varphi$ in $\PE$ with codomain $|\varphi|=T$, and every structured source $\psi$ in $\PdE$ with domain $|\psi|=T$. The full subcategory $\sI\CE$ of $\PE$ contains the structured sinks fixed under the correspondence.

Recall that a full subcategory $\CE$ is finally dense in the concrete category $\CC$ over $\CB$ if every $\CC$-object is the codomain of some $(|\text{-}|)$-final structured sink with domains in $\CE$. Initial density is the dual concept.

\begin{cor}
Let $\CE$ be a full subcategory of a concrete category $\CC$ over $\CB$. Then $\CC$ is concretely equivalent to $\sI\overline{\CE}$ if, and only if, $\CC$ is topological over $\CB$ and $\CE$ is finally and initially dense in $\CC$.
\end{cor}

\begin{proof}
Final density amounts to density as a $\QB$-category, and initial density to codensity. Hence, Corollary \ref{dense_codense_subcategory} applies.
\end{proof}

A topological category $\CC$ over $\CB$ is called the \emph{MacNeille completion} of its full subcategory $\CE$ if $\CE$ is finally and initially dense in $\CC$ (see \cite{Adamek1990}). As the description of $\sI\CE$ above shows, in that case the category $\CC$ may be built constructively from $\CE$. The smaller $\CE$ may be chosen the stronger the benefit of this fact becomes.

\begin{exmp} {\rm\cite{Davey2002}}
For a complete lattice $L$ (considered as a small topological category over the terminal category ${\bf 1}$) with no infinite chains, the subsets of join-irreducible elements $J$ and of meet-irreducible elements $M$ are respectively finally dense and initially dense in $L$. Let $\Phi:J\oto M$ be the order relation inherited from $L$ between ordered sets $J,M$ considered as categories concrete over ${\bf 1}$, then the assignment $x\mapsto\da x\cap J$ gives rise to the isomorphism $L\cong\sI\Phi$. The union $J\cup M$ is both finally and initially dense in $L$, which is therefore its MacNeille completion.
\end{exmp}

\begin{exmp}
In the topological category $\Ord$ over $\Set$ of preordered sets and their monotone maps, the full subcategory with the two-element chain ${\bf 2}$ as its only object is both finally and initially dense. Indeed, for every $X\in\ob\Ord$, the sink $\Ord({\bf 2},X)$ described by all pairs $(x,y)$ with $x\leq y$ in $X$) is final, and the structured source $\psi$ with $\psi_{\bf 2}=\Ord(X,{\bf 2})$ (described by all downsets in $X$) is initial; in fact, already the source of the principal downsets $\da x$, $x\in X$ is initial.
\end{exmp}

\begin{exmp} {\rm\cite{Herrlich1992}}
The category $\Rel$ of sets that comes equipped with an arbitrary relation on them and their relation-preserving maps as morphisms is topological over $\Set$. The full subcategory with $A = \{0,1\}$ equipped with the relation $\{(0,1)\}$ as its only object is finally dense; indeed, for every $X\in\ob\Rel$, the sink $\Rel(A,X)$ describes all related pairs in $X$ and is therefore final. An initially dense one-object full subcategory may be given by equipping the set $B = \{0,1\}$ with the relation $\{(0,0),(0,1),(1,0)\}$; indeed, for every $X\in\ob\Rel$, initiality of the source $\Rel(X,B)$ is easily established.
\end{exmp}

While only few topological categories over $\Set$ contain a one-object subcategory that is both finally and initially dense, there is a general type of topological categories over $\Set$ that admits a one-object initially dense subcategory. Indeed, if $\CQ$ is a (small) quantale, i.e., a one-object quantaloid, then $\CQ$ becomes a $\CQ$-category whose objects are the elements of $\CQ$, and whose hom-arrows are given by
$$\CQ(u,v)=v\lda u\quad(u,v\in\CQ).$$

\begin{prop} \label{Q_initial_dense_QCat}
For a quantale $\CQ$, the functor $\ob:\QCat\to\Set$ is topological, and the full subcategory with $\CQ$ as its only object is initially dense in $\QCat$.
\end{prop}

\begin{proof}
For the topologicity assertion, see \cite[Theorem III.3.1.3]{Hofmann2014}, and for the fact that $\{\CQ\}$ is initially dense in $\QCat$, see \cite[Exercise III.1.H]{Hofmann2014}. Indeed, given a small $\CQ$-category $\CE$, initiality of the source
$$\CE(X,-):\CE\to \CQ\quad(X\in\ob\CE)$$
is easily verified.
\end{proof}

\begin{rem}
As has been shown in \cite{Shen2015a}, the first assertion of Proposition \ref{Q_initial_dense_QCat} easily generalizes from quantales to quantaloids: $\QCat$ is topological over $\Set/\ob\CQ$ for any small quantaloid $\CQ$.
\end{rem}

For $\CQ={\bf 2}$ the two-element chain (with $\circ$ given by meet), Proposition \ref{Q_initial_dense_QCat} reproduces the initiality assertion for ${\bf 2}$ in $\Ord={\bf 2}\text{-}\Cat$. For $\CQ=([0,\infty],\geq)$ (with $\circ$ given by $+$), $\QCat=\Met$ is Lawvere's category of generalized metric spaces $(X,d)$ (with $d:X\times X\to[0,\infty]$ satisfying $d(x,x)=0$ and $d(x,z)\leq d(x,y)+d(y,z)$ for all $x,y,z\in X$) and their non-expanding maps $f:(X,d)\to(Y,e)$ (with $e(f(x),f(y))\leq d(x,y)$ for all $x,y\in X$).

\begin{exmp}
$([0,\infty],h)$ with
$$h(r,s)=\begin{cases}
s-r, & \text{if}\ r\leq s<\infty,\\
0, & \text{if}\ s\leq r,\\
\infty, & \text{if}\ r<s=\infty
\end{cases}$$
is initially dense in $\Met$ but not finally dense.
\end{exmp}

\section{Total $\CQ$-categories as injective objects} \label{total_injective}

It is well-known (see \cite{Stubbe2013a}) that small total $\CQ$-categories are characterized as the injective objects in $\QCat$. For a large $\CQ$-category $\CE$, the standard proof which employs $\PE$ as a test object, has to be modified as $\PE$ may be illegitimately large. We therefore give a modified proof that is valid also in the large case.

\begin{thm}[Stubbe \cite{Stubbe2013a}] \label{total_injective_fff}
A $\CQ$-category is total if, and only if, it is injective\footnote{Strictly speaking, ``injective'' should read ``\emph{quasi-injective}'', since, in the proof below, generally we obtain $H$ only with $HG\cong F$, not $HG=F$.} in $\QCAT$ w.r.t. fully faithful $\CQ$-functors.
\end{thm}

\begin{proof}
For $\CE$ total and $\CQ$-functors $F:\CC\to\CE$, $G:\CC\to\CD$ with $G$ fully faithful we must find $H:\CD\to\CE$ with $HG\cong F$. With $\widehat{G_{\nat}}$ denoting the transpose of $G_{\nat}:\CC\oto\CD$, we may define
$$H=(\CD\to^{\widehat{G_{\nat}}}\PC\to^{F_!}\PE\to^{\sup}\CE).$$
One then has, for all $X\in\ob\CE$,
\begin{align*}
HGX&=\sup F_!(G_{\nat}(-,GX))\\
&=\sup F_!(\CD(G-,GX))\\
&=\sup F_!\sY_{\CC}X&(G\ \text{fully faithful})\\
&=\sup\sY_{\CE}FX&(\sY:1\to\sP\ \text{natural})\\
&\cong FX.
\end{align*}

Conversely, for a $\CQ$-category $\CE$ and $\varphi\in\ob\PE$, in order to find the supremum of $\varphi$, one may exploit its injectivity in $\QCAT$ on the fully faithful $\CQ$-functor $\sY_{\CE}:\CE\to\CD$, where $\CD$ is the $\CQ$-subcategory of the $\CQ$-(meta)category $\PE$ with $\ob\CD=\{\sY_{\CE}X\mid X\in\ob\CE\}\cup\{\varphi\}$. Thus one obtains a $\CQ$-functor $H:\CD\to\CE$ with $H\sY_{\CE}=1_{\CE}$. Note that for all $X\in\ob\CE$,
\begin{align*}
\CE(H\varphi,X)&=\PE(\sY_{\CE}H\varphi,\sY_{\CE}X)&\text{(Yoneda Lemma)}\\
&=\sY_{\CE}X\lda\CE(-,H\varphi)\\
&=\sY_{\CE}X\lda\CE(H\sY_{\CE}-,H\varphi)&(H\sY_{\CE}=1_{\CE})\\
&\leq\sY_{\CE}X\lda\CD(\sY_{\CE}-,\varphi)\\
&=\sY_{\CE}X\lda\varphi&\text{(Yoneda Lemma)}\\
&=\CD(\varphi,\sY_{\CE}X)\\
&\leq\CE(H\varphi,H\sY_{\CE}X)\\
&=\CE(H\varphi,X).&(H\sY_{\CE}=1_{\CE})
\end{align*}
Therefore $\CE(H\varphi,X)=\sY_{\CE}X\lda\varphi=\PE(\varphi,\sY_{\CE}X)$ for all $X\in\ob\CE$ and, consequently, $H\varphi$ is the supremum of $\varphi$.
\end{proof}

Exploiting Theorem \ref{total_injective_fff} for $\CQ=\QB$, where $\CB$ is an ordinary category, one reproduces a classical result of categorical topology:

\begin{cor}[Br{\" u}mmer-Hoffmann \cite{Brummer1976}]
For a topological category $\CE$ over $\CB$ and concrete functors $F:\CC\to\CE$, $G:\CC\to\CD$ over $\CB$ with $G$ fully faithful, there is a concrete functor $H:\CD\to\CE$ with $HG\cong F$ concretely isomorphic. Conversely, this property characterizes topologicity of $\CE$ over $\CB$.
\end{cor}





\end{document}